\documentclass{amsart}

\usepackage{amsmath,amssymb}
\usepackage{amsthm}
\usepackage{amsrefs}
\usepackage{indentfirst}
\usepackage{mathrsfs}
\usepackage{hyperref}
\usepackage{xcolor}
\usepackage[margin=1.4in]{geometry}
\usepackage{nicefrac}
\usepackage{graphics}
\usepackage{tikz}
\usepackage{bm}

\newcommand{\bR}{\mathbb{R}}
\newcommand{\bfx}{\mathbf{x}}
\newcommand{\bfy}{\mathbf{y}}
\newcommand{\bmeta}{\bm{\eta}}
\newcommand{\bfe}{\mathbf{e}}
\newcommand{\hbfx}{\hat{\mathbf{x}}}

\newcommand{\calP}{\mathcal{P}}
\newcommand{\bfs}{\mathbf{s}}
\newcommand{\bfh}{\mathbf{h}}
\newcommand{\im}{\text{im}}

\newlength{\leftstackrelawd}
\newlength{\leftstackrelbwd}
\def\leftstackrel#1#2{\settowidth{\leftstackrelawd}%
{${{}^{#1}}$}\settowidth{\leftstackrelbwd}{$#2$}%
\addtolength{\leftstackrelawd}{-\leftstackrelbwd}%
\leavevmode\ifthenelse{\lengthtest{\leftstackrelawd>0pt}}%
{\kern-.5\leftstackrelawd}{}\mathrel{\mathop{#2}\limits^{#1}}}

\numberwithin{equation}{section}
\newtheorem{theorem}{Theorem}[section]
\newtheorem{lemma}[theorem]{Lemma}

\newtheorem{corollary}[theorem]{Corollary}
\newtheorem{example}[theorem]{Example}
\allowdisplaybreaks[4]

\title[Representation Theorem for Multivariable Totally Symmetric Functions]{Representation Theorem for Multivariable Totally Symmetric Functions}
\author{Chongyao Chen}
\address{(CC) Department of Mathematics, Duke University, Box 90320, Durham, NC 27708.}
\email{cychen@math.duke.edu}
\author{Ziang Chen}
\address{(ZC) Department of Mathematics, Massachusetts Institute of Technology, 77 Massachusetts Avenue, Cambridge, MA 02139.}
\email{ziang@mit.edu}
\author{Jianfeng Lu}
\address{(JL) Departments of Mathematics, Physics, and Chemistry, Duke University, Box 90320, Durham, NC 27708.}
\email{jianfeng@math.duke.edu}
\date{\today}

\thanks{The work of ZC and JL is supported in part by National Science Foundation via awards DMS-2012286 and DMS-2309378.}

\begin{document}

\begin{abstract}
In this work, we establish a representation theorem for multivariable totally symmetric functions: a multisymmetric continuous function must be the composition of a continuous function and a set of generators of the multisymmetric polynomials. We then study the singularity and geometry of the generators, and show that the regularity may become worse after applying the decomposition.
\end{abstract}

\maketitle

\section{Introduction}

Symmetric and anti-symmetric functions play important roles in physics and chemistry, especially in representing many-body systems (see e.g., \cites{thouless2014quantum,dickhoff2008many, csanyi2000tensor, heugel2019classical}). For more efficient modeling and computation, researchers have been investigating the representation and approximation results for (anti-)symmetric functions \cites{bachmayr2021polynomial, zaheer2017deep, hutter2020representing, huang2021geometry, han2019universal, wagstaff2019limitations}.

We study the representation theorem for symmetric functions: given $d\geq 1$, $n\geq 1$, and $\Omega\subset\bR^d$, a function $f:\Omega^n\rightarrow \bR$ is \emph{totally symmetric}, or \emph{symmetric}, if
\begin{equation}\label{eq:symmetric}
    f(\bfx_{\sigma(1)},\bfx_{\sigma(2)},\dots,\bfx_{\sigma(n)}) = f(\bfx_1,\bfx_2,\dots,\bfx_n),
\end{equation}
for any permutation $\sigma\in S_n$ and any $\bfx_i\in\Omega,\ i=1,2,\dots,n$. It is proved in \cite{zaheer2017deep} that when $d=1$ and $\Omega=[0,1]$, a continuous symmetric function $f:[0,1]^n\rightarrow\bR$ can be decomposed as
\begin{equation}\label{eq:decomp1d}
    f(x_1,x_2,\dots,x_n) = g\left(\sum_{i=1}^n (1,x_i,x_i^2,\dots,x_i^n)\right) =  g\left(n,\sum_{i=1}^n x_i,\sum_{i=1}^n x_i^2, \dots, \sum_{i=1}^n x_i^n \right),
\end{equation}
where $g$ is continuous. One can see that the embedding dimension in \eqref{eq:decomp1d} is essentially $n$ since the constant $1$ can be dropped. It is proved in \cite{wagstaff2019limitations} that such embedding is necessary and hence optimal to exactly represent all continuous symmetric function. Even if the case $d=1$ is well understood, the proof in \cite{zaheer2017deep} cannot be generalized to $d>1$, which will be elaborated in Section~\ref{sec:represent_thm}. Thus, it remains unclear in the previous literature whether similar representation theorem holds for the multivariable case $d>1$, though several approximation results have been established \cites{zaheer2017deep, han2019universal, bachmayr2021polynomial}. There is also work on representing partial permutation-invariant functions \cite{gui2021pine}, but still in the sense of approximation. 

Our aim of this work is to establish an exact representation theorem for multivariable totally symmetric functions (multisymmetric functions in short) and further characterizing the representation. More specifically, we prove that a result similar to \eqref{eq:decomp1d} for arbitrary $d\geq 1$, and show that the regularity of $g$ may be worse than that of $f$, even if both of them are continuous. 
We hope this clarifies the question on representation and helps better understand approximating symmetric functions. 

\section{Representation Theorem}
\label{sec:represent_thm}

Let $\calP_{\text{sym}}^{d,n}(\bR)$ be the $\bR$-algebra consisting of all multisymmetric polynomials with real coefficients, i.e., real polynomials in $\bfx = (\bfx_1,\bfx_2,\dots,\bfx_n)\in(\bR^d)^n$ satisfying \eqref{eq:symmetric}. The algebra of multisymmetric polynomials is well-studied, even when the coefficients are in a general ring (see \cite{briand2004algebra} and references therein). According to \cite{briand2004algebra} (also see \cite{weyl1946classical}), similar to the setting of $d=1$, $\calP_{\text{sym}}^{d,n}(\bR)$ can be generated by
\begin{equation}\label{eq:generator}
    \eta_{\bfs}(\bfx_1,\bfx_2,\dots,\bfx_n) := \sum_{i=1}^n x_{i,1}^{s_1}x_{i,2}^{s_2}\cdots x_{i,d}^{s_d},\quad 0\leq  s_1+s_2+\cdots+s_d\leq n,
\end{equation}
known as multisymmetric power sums, where $\bfs := (s_1,s_2,\dots,s_d)$ and $\bfx_i := (x_{i,1},x_{i,2},\dots,x_{i,d})$, and it has other sets of generators such as the elementary multisymmetric polynomials.

Our main representation theorem is stated as below.
\begin{theorem}\label{thm:main}
Given $d\geq 1$, $n\geq 1$, and a compact subset $\Omega\subset\bR^d$, suppose that $f:\Omega^n\rightarrow \bR$ is totally symmetric and continuous and that $\eta_1,\eta_2,\dots,\eta_m$ generate $\calP_{\text{sym}}^{d,n}(\bR)$ as $\bR$-algebra. Then there exists a unique continuous function $g:\bmeta(\Omega^n)\rightarrow \bR$ such that 
\begin{equation*}
    f(\bfx_1,\bfx_2,\dots,\bfx_n) = g(\bmeta(\bfx_1,\bfx_2,\dots,\bfx_n)),\quad \forall~\bfx_1,\bfx_2,\dots,\bfx_n\in\Omega,
\end{equation*}
where $\bmeta = (\eta_1,\eta_2,\dots,\eta_m)$ and the topology of $\bmeta(\Omega)$ is induced from $\bR^m$.
\end{theorem}

Theorem~\ref{thm:main} generalizes \cite{zaheer2017deep}*{Theorem 7} to the multivariable case and general generators (not limited to power sums). A theorem similar to Theorem~\ref{thm:main} was claimed in \cite{hutter2020representing}*{Theorem 6}, but no proof is given; this work fills the gap. Our proof is inspired by the proof of \cite{zaheer2017deep}*{Theorem 7}, while to make the proof work for arbitrary dimension, we need to argue that $\bmeta(\Omega^n)$ is homeomorphic to the quotient space $\Omega^n/S_n$, not a subspace of $\Omega^n$ as in the proof of \cite{zaheer2017deep}*{Theorem 7}, which will be discussed in details after we present the proof.

\begin{proof}[Proof of Theorem~\ref{thm:main}]
There is a natural group action of $S_n$ on $\Omega^n$:
\begin{equation}\label{eq:action}
    \sigma\ast(\bfx_1,\bfx_2,\dots,\bfx_n) := (\bfx_{\sigma(1)},\bfx_{\sigma(2)},\dots,\bfx_{\sigma(n)}),\quad \forall~\sigma\in S_n,\ \bfx_1,\bfx_2,\dots,\bfx_n\in \Omega,
\end{equation}
and we denote $\pi:\Omega^n\rightarrow \Omega^n/S_n$ as the quotient map. Since $f$ is totally symmetric, i.e., \eqref{eq:symmetric} holds for any $\sigma\in S_n$ and $\bfx_1,\bfx_2,\dots,\bfx_n\in \Omega$, there exists a unique continuous function $\Tilde{f}:\Omega^n/S_n\rightarrow\bR$ such that $f = \Tilde{f}\circ \pi$. In addition, the definition of $\bmeta$ immediately implies that
\begin{equation*}
    \bmeta(\sigma\ast (\bfx_1,\bfx_2,\dots,\bfx_n)) = \bmeta(\bfx_1,\bfx_2,\dots,\bfx_n),\quad \forall~\sigma\in S_n,\ \bfx_1,\bfx_2,\dots,\bfx_n\in \Omega.
\end{equation*}
Therefore, there exists a unique continuous function $\Tilde{\bmeta}:\Omega^n/S_n\rightarrow \bmeta(\Omega^n)$ such that $\bmeta = \Tilde{\bmeta}\circ\pi$. 

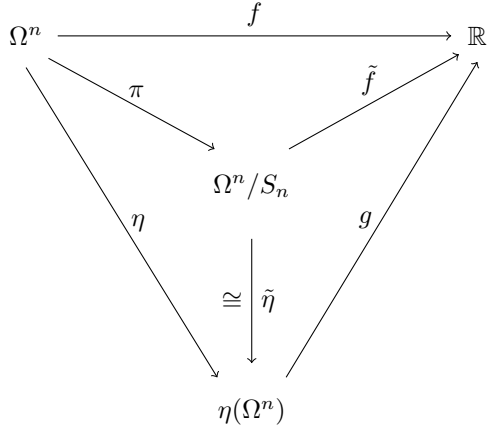
\begin{figure}
    \centering
    \begin{tikzpicture}[
	    constraintb/.style={circle},
			]
			
	    \draw (-3,0) node[constraintb] (x1) {$\Omega^n$};
	    \draw (3,0) node[constraintb] (x2) {$\bR$};
	    \draw (0,-2) node[constraintb] (x3) {$\Omega^n/S_n$};
	    \draw (0,-5) node[constraintb] (x4) {$\bmeta(\Omega^n)$};
			
		\draw[->] (x1.east) -- (x2.west) node[above, pos = 1/2] {$f$};
		\draw[->] (x1.south east) -- (x3.north west) node[above, pos = 0.53] {$\pi$};
		\draw[->] (x3.north east) -- (x2.south west) node[above, pos = 0.47] {$\Tilde{f}$};
		\draw[->] (x1.south) -- (x4.north west) node[right, pos = 1/2] {$\bmeta$};
		\draw[->] (x4.north east) -- (x2.south) node[left, pos = 1/2] {$g$};
		\draw[->] (x3.south) -- (x4.north) node[right, pos = 1/2] {$\Tilde{\bmeta}$} node[left, pos = 1/2] {$\cong$};
		\end{tikzpicture}
    \caption{Commutative diagram for the proof of Theorem~\ref{thm:main}.}
    \label{comm_diag}
\end{figure}

It is clear that $\Tilde{\bmeta}:\Omega^n/S_n\rightarrow \bmeta(\Omega^n)$ is surjective. We then prove that it is also injective. Consider any $\pi(\bfx_1,\bfx_2,\dots,\bfx_n),\pi(\bfx_1',\bfx_2',\dots,\bfx_n')\in \Omega^n/S_n$ with $\Tilde{\bmeta}(\pi(\bfx_1,\bfx_2,\dots,\bfx_n)) = \Tilde{\bmeta}(\bfx_1',\bfx_2',\dots,\bfx_n')$, or equivalently, $\bmeta(\bfx_1,\bfx_2,\dots,\bfx_n)=\bmeta(\bfx_1',\bfx_2',\dots,\bfx_n')$. Note that entries of $\bmeta$ are generators of $\calP_{\text{sym}}^{d,n}(\bR)$. Thus, $p(\bfx_1,\bfx_2,\dots,\bfx_n)=p(\bfx_1',\bfx_2',\dots,\bfx_n')$ holds for any multisymmetric polynomial $p\in \calP_{\text{sym}}^{d,n}(\bR)$. By Lemma~\ref{lem:separate_pt}, there must exist some permutation $\sigma\in S_n$ such that $(\bfx_1',\bfx_2',\dots,\bfx_n') = \sigma\ast (\bfx_1,\bfx_2,\dots,\bfx_n)$, i.e., $\pi(\bfx_1,\bfx_2,\dots,\bfx_n)=\pi(\bfx_1',\bfx_2',\dots,\bfx_n')$, which proves the injectivity.

Since $\Tilde{\bmeta}:\Omega^n/S_n\rightarrow \bmeta(\Omega^n)$ is bijective, its inverse is well-defined. We then show that $\Tilde{\bmeta}^{-1}:\bmeta(\Omega^n)\rightarrow \Omega^n/S_n$ is also continuous, which is equivalent to show that $\Tilde{\bmeta}:\Omega^n/S_n\rightarrow \bmeta(\Omega^n)$ is a closed map. Consider any closed subset $E\subset \Omega^n/S_n$. The continuity of $\pi$ guarantees that $\pi^{-1}(E)$ is closed subset of $\Omega^n$. Furthermore, $\pi^{-1}(E)$ is compact due to the compactness of $\Omega^n$. Note that $\bmeta:\Omega^n\rightarrow \bmeta(\Omega^n)$ is continuous. Thus, both $\bmeta(\Omega^n)$ and $\bmeta(\pi^{-1}(E)) = \Tilde{\bmeta}(E)$ are compact. Then one can conclude that $\Tilde{\bmeta}(E)$ is a closed subset of $\bmeta(\Omega^n)$ and hence that $\Tilde{\bmeta}:\Omega^n/S_n\rightarrow \bmeta(\Omega^n)$ is closed.

Combining all these properties of $\Tilde{\bmeta}:\Omega^n/S_n\rightarrow \bmeta(\Omega^n)$, one can conclude it is indeed a homeomorphism. We can finally complete the proof by setting $g = \Tilde{\bmeta}^{-1}\circ\Tilde{f}$, since it follows from the commutative diagram in Figure~\ref{comm_diag} that $f = g\circ\bmeta$.
\end{proof}

\begin{lemma}\label{lem:separate_pt}
For any $(\bfx_1,\bfx_2,\dots,\bfx_n),(\bfx_1',\bfx_2',\dots,\bfx_n')\in(\bR^d)^n$, if
\begin{equation}\label{no_perm}
    (\bfx_1',\bfx_2',\dots,\bfx_n') \neq \sigma\ast (\bfx_1,\bfx_2,\dots,\bfx_n),\quad \forall~\sigma\in S_n,
\end{equation}
then there exists some multisymmetric polynomial $p\in \calP_{\text{sym}}^{d,n}(\bR)$ such that $p(\bfx_1,\bfx_2,\dots,\bfx_n)\neq p(\bfx_1',\bfx_2',\dots,\bfx_n')$.
\end{lemma}

\begin{proof}
Denote $\{\hbfx_1,\hbfx_2,\dots,\hbfx_t\}:=\{\bfx_1,\bfx_2,\dots,\bfx_n\}\cup\{\bfx_1',\bfx_2',\dots,\bfx_n'\}$.
Define the counters $c_j := \sum_{i=1}^n \mathbb{I}(\hbfx_j = \bfx_i)$ and $c_j' := \sum_{i=1}^n \mathbb{I}(\hbfx_j = \bfx_i')$, for $j=1,2,\dots,t$, where $\mathbb{I}(\cdot)$ is the indicator function. It follows from \eqref{no_perm} that $(c_1,c_2,\dots,c_t)\neq (c_1',c_2',\dots,c_t')$. Thus, there exist $z_1,z_2,\dots,z_t\in\bR$ such that $\sum_{j=1}^t c_j z_j \neq \sum_{j=1}^t c_j' z_j$. Let $q:\bR^d\rightarrow\bR$ be a polynomial satisfying $q(\hbfx_j) = z_j,\ j=1,2,\dots,t$ (the existence of such $q$ is guaranteed by polynomial interpolation) and let $p\in \calP_{\text{sym}}^{d,n}(\bR)$ be defined via
\begin{equation*}
    p(\bfy_1,\bfy_2,\dots,\bfy_n) := \sum_{i=1}^n q(\bfy_i),\quad \forall~\bfy_1,\bfy_2,\dots,\bfy_n\in\bR^d.
\end{equation*}
Then it holds that
\begin{align*}
    p(\bfx_1,\bfx_2,\dots,\bfx_n) & = \sum_{i=1}^n q(\bfx_i) = \sum_{j=1}^t c_j q(\hbfx_j) = \sum_{j=1}^t c_j z_j \\
    & \neq \sum_{j=1}^t c_j' z_j = \sum_{j=1}^t c_j' q(\hbfx_j) = \sum_{i=1}^n q(\bfx_i') = p(\bfx_1',\bfx_2',\dots,\bfx_n'),
\end{align*}
which completes the proof.
\end{proof}

Let us remark on the connection and difference between our work and \cite{zaheer2017deep} that establishes a special case of Theorem~\ref{thm:main}, say $d=1$. One key step in \cite{zaheer2017deep} is to prove the homeomorphism $\bmeta(\Omega^n)\cong \mathcal{X} := \{(x_1,x_2,\dots,x_n)\in\Omega^n:x_1\leq x_2\leq \cdots\leq x_n\}$ that is actually equivalent to $\bmeta(\Omega^n)\cong \Omega^n/S_n$ when $d=1$. However, the result $\bmeta(\Omega^n)\cong \mathcal{X}$ may not hold high dimensions. As an explicit example, let us consider $d=n=2$ and define $\mathcal{X} = \{(\bfx_1,\bfx_2)\in \Omega^2 : \bfx_1\preceq \bfx_2\}$ that is equipped with the topology induced from the Euclidean topology, with $\Omega = [0,1]^2$ and $\preceq$ being the lexicographic order. One observation is that $\big((1-\frac{1}{n},1), (1,0)\big)\in\mathcal{X}$ for any $n\in\mathbb{N}_+$ while their limiting point $((1,1),(1,0))\notin\mathcal{X}$. This means that $\mathcal{X}$ is not closed, and hence cannot be homeomorphic to $\bmeta(\Omega^n)$. To resolve this issue, we have to work with the quotient space $\Omega^n/S_n$ instead of $\mathcal{X}\subset\Omega^n$, as in the proof of Theorem~\ref{thm:main}.

The next corollary generalizes Theorem~\ref{thm:main} in the sense that the result can hold in the whole space $\bR^d$, not just a compact subset $\Omega\subset\bR^d$.

\begin{corollary}\label{cor:wholespace}
Given $d\geq 1$, $n\geq 1$, suppose that $f:(\bR^d)^n\rightarrow \bR$ is totally symmetric and continuous and that $\eta_1,\eta_2,\dots,\eta_m$ generate $\calP_{\text{sym}}^{d,n}(\bR)$ as $\bR$-algebra. Then there exists a unique continuous function $g:\bmeta((\bR^d)^n)\rightarrow \bR$ such that 
\begin{equation*}
    f(\bfx_1,\bfx_2,\dots,\bfx_n) = g(\bmeta(\bfx_1,\bfx_2,\dots,\bfx_n)),\quad \forall~\bfx_1,\bfx_2,\dots,\bfx_n\in\bR^d,
\end{equation*}
where $\bmeta = (\eta_1,\eta_2,\dots,\eta_m)$ and the topology of $\bmeta(\bR^d)$ is induced from $\bR^m$.
\end{corollary}

\begin{proof}
According to Theorem~\ref{thm:main}, for any $r\in(0,+\infty)$, there uniquely exists continuous $g_r: \bmeta(\Omega_r^d)\rightarrow\bR$ such that
\begin{equation}\label{eq:gr}
    f(\bfx_1,\bfx_2,\dots,\bfx_n) = g_r(\bmeta(\bfx_1,\bfx_2,\dots,\bfx_n)),\quad \forall~\bfx_1,\bfx_2,\dots,\bfx_n\in\Omega_r,
\end{equation}
where $\Omega_r\subset \bR^d$ is the closed $\ell_2$-ball centered at the origin with radius $r$. It is clear by \eqref{eq:gr} that for any $r,r'\in(0,+\infty)$, $g_r$ and $g_{r'}$ coincide on $\bmeta((\Omega_r\cap\Omega_{r'})^n)$. Thus, the desired function $g:\bmeta((\bR^d)^n)\rightarrow \bR$ can be well-defined via
\begin{equation*}
    g(\bmeta(\bfx_1,\bfx_2,\dots,\bfx_n)) = g_r(\bmeta(\bfx_1,\bfx_2,\dots,\bfx_n)), \quad \text{if }\bfx_1,\bfx_2,\dots,\bfx_n\in\Omega_r.
\end{equation*}
In order to prove the continuity of $g$, it suffices to show that for any bounded subset $Y\subset\bR^m$, $g|_{\bmeta((\bR^d)^n)\cap Y}$ is continuous. By Lemma~\ref{lem:bdd_preimage}, $\bmeta^{-1}(Y)$ is bounded and is hence contained in $\Omega_r^n$ for some $r>0$. Thus, the continuity of $g|_{\bmeta((\bR^d)^n)\cap Y}$ is implied directly by the continuity of $g_r$, which finishes the proof.
\end{proof}

\begin{lemma}\label{lem:bdd_preimage}
Suppose that $\eta_1,\eta_2,\dots,\eta_m$ generate $\calP_{\text{sym}}^{d,n}(\bR)$ as $\bR$-algebra. Then for any bounded subset $Y\subset\bR^m$, $\bmeta^{-1}(Y)$ must be bounded in $(\bR^d)^n$, where $\bmeta = (\eta_1,\eta_2,\dots,\eta_m)$.
\end{lemma}

\begin{proof}
There exists a polynomial $p:\bR^m\rightarrow\bR$, such that
\begin{equation*}
    \sum_{i=1}^n \sum_{j=1}^d x_{i,j}^2 = p(\eta_1(\bfx),\eta_2(\bfx),\dots,\eta_m(\bfx)),\quad \forall~\bfx = (\bfx_1,\bfx_2,\dots,\bfx_n)\in (\bR^d)^n,
\end{equation*}
where $\bfx_i = (x_{i,1},x_{i,2},\dots,x_{i,d})$ for $i=1,2,\dots,n$. Thus, it holds that
\begin{equation*}
    \left\{\sum_{i=1}^n \sum_{j=1}^d x_{i,j}^2 : \bfx\in \bmeta^{-1}(Y)\right\} = p(Y)
\end{equation*}
is bounded, which implies the boundedness of $\bmeta^{-1}(Y)$.
\end{proof}

\section{Singularity of Symmetric Decomposition}

In this section, we discuss the singularity of the symmetric decomposition $f = g\circ \bmeta$. In Theorem~\ref{thm:main} and Corollary~\ref{cor:wholespace}, we prove that $g$ inherits the continuity property of $f$. However, it may not be true that $g$ has the same regularity as $f$. The reason is that $\bmeta$ is singular at 
\begin{equation*}
    \mathbf{Sing}_{d,n}:=\left\{\bfx =(\bfx_1,\bfx_2,\dots,\bfx_n)\in (\bR^{d})^n\ |\ \exists~i_1\neq j_2,\ \bfx_{i_1}=\bfx_{i_2}\right\}.
\end{equation*} 
This result is established in the following theorem, where we remark that the number of generators satisfies $m\geq \binom{n+d}{d} > nd$ by \eqref{eq:generator}.

\begin{theorem}\label{thm:singular}
Let $\eta_1,\eta_2,\dots,\eta_m$ generate $\calP_{\text{sym}}^{d,n}(\bR)$ as $\bR$-algebra and let $\bmeta = (\eta_1,\eta_2,\dots,\eta_m)$. Then the locus, where Jacobian matrix $J\bmeta(\bfx)\in\bR^{m\times nd}$ is column-rank-deficient, is the set $\mathbf{Sing}_{d,n}$ defined above. 
\end{theorem}

\begin{proof}
We first assume that $\bfx_{i_1} = \bfx_{i_2}$ for some $1\leq i_1 < i_2 \leq n$, and show that $J\bmeta(\bfx)$ is column-rank-deficient. Consider any $j\in\{1,2,\dots,d\}$ and we have that
\begin{align*}
    \frac{\partial\bmeta(\bfx)}{\partial x_{i_1,j}} & = \lim_{t\rightarrow 0} \frac{\bmeta(\bfx_1,\dots,\bfx_{i_1-1}, \bfx_{i_1} + t \bfe_j, \bfx_{i_1+1},\dots, \bfx_{i_2-1},\bfx_{i_2},\bfx_{i_2+1},\dots,\bfx_n) - \bmeta(\bfx)}{t} \\
    & = \lim_{t\rightarrow 0} \frac{\bmeta(\bfx_1,\dots,\bfx_{i_1-1}, \bfx_{i_1}, \bfx_{i_1+1},\dots, \bfx_{i_2-1},\bfx_{i_2} + t \bfe_j,\bfx_{i_2+1},\dots,\bfx_n) - \bmeta(\bfx)}{t} \\
    & = \frac{\partial\bmeta(\bfx)}{\partial x_{i_2,j}},
\end{align*}
where $\bfe_j$ is a vector in $\bR^d$ with the $j$-th entry being $1$ and other entries being $0$. Thus, $J\bmeta(\bfx)$ has two identical columns and is hence column-rank-deficient.

Then we assume that $\bfx_{i_1}\neq \bfx_{i_2}$ for all $1\leq i_1<i_2\leq n$ and prove that $J\bmeta(\bfx)$ is full-column-rank. There exists an invertible matrix $A\in\bR^{d\times d}$ such that $\bfy_i = A\bfx_i,\ i=1,2,\dots,n$ satisfy that $y_{i_1,j}\neq y_{i_2,j}$ for any $1\leq i_1<i_2\leq n$ and any $1\leq j\leq d$. Define $\bm{\phi}:\bR^{nd}\to\bR^{nd}$:
\begin{equation*}
    \bm{\phi}(\bfy) = (\phi_{s,j}(\bfy))_{1\leq s\leq n, 1\leq j\leq d},
\end{equation*}
where
\begin{equation*}
    \phi_{s,j}(\bfy) = \phi_{s,j}(\bfy_1,\bfy_2,\dots,\bfy_n) = \frac{1}{s} \sum_{i=1}^n y_{i,j}^s. 
\end{equation*}
It can be computed that
\begin{equation*}
    J\bm{\phi}(\bfy) = \text{diag}\left(\begin{pmatrix}
        1 & 1 & \cdots & 1 \\
        y_{1,1} & y_{2,1} & \cdots & y_{n,1} \\
        \vdots & \vdots & \ddots & \vdots \\
        y_{1,1}^{n-1} & y_{2,1}^{n-1} & \cdots & y_{n,1}^{n-1}
    \end{pmatrix},\dots, \begin{pmatrix}
        1 & 1 & \cdots & 1 \\
        y_{1,d} & y_{2,d} & \cdots & y_{n,d} \\
        \vdots & \vdots & \ddots & \vdots \\
        y_{1,d}^{n-1} & y_{2,d}^{n-1} & \cdots & y_{n,d}^{n-1}
    \end{pmatrix}\right)\in\bR^{nd\times nd},
\end{equation*}
which is invertible since $y_{1,j},y_{2,j},\dots,y_{n,j}$ are distinct for any $j\in\{1,2,\dots,d\}$. Then define
\begin{equation*}
    \bm{\varphi}(\bfx) = \bm{\varphi}(\bfx_1,\bfx_2,\dots,\bfx_n) = \bm{\phi}(A\bfx_1,A\bfx_2,\dots,A\bfx_n),
\end{equation*}
and one can see that each entry of $\bm{\varphi}$ is a multisymmetric polynomial of $\bfx=(\bfx_1,\bfx_2,\dots,\bfx_n)$. Since $\eta_1,\eta_2,\dots,\eta_m$ generate $\calP_{\text{sym}}^{d,n}(\bR)$, there exists a polynomial map $\bfh:\bR^m\to \bR^{nd}$ such that
\begin{equation*}
    \bm{\varphi} = \bfh \circ \bmeta.
\end{equation*}
Differentiating both sides yields that
\begin{equation*}
    J\bm{\phi}(\bfy) \text{diag}(A,A,\dots,A) = J\bfh(\bmeta(\bfx)) J\bmeta(\bfx).
\end{equation*}
Note that $J\bm{\phi}(\bfy), \text{diag}(A,A,\dots,A)\in\bR^{nd\times nd}$ are both invertible. One can thus conclude that $J\eta(\bfx)\in\bR^{m\times nd}$ is of full-column-rank.
\end{proof}

Theorem~\ref{thm:singular} establishes the location where $\bmeta$ is singular. In general, the geometry of $\im(\bmeta) = \bmeta((\bR^{d})^n)$ seems complicated and we leave the general characterization for future works. In the following, we consider the special case with $n=d=2$ and $\bmeta$ constructed using generators in \eqref{eq:generator}, for which the geometry can be understood rather explicitly.
After omitting the constant polynomial in the set of generators, $\bmeta$ can be written as

\begin{equation*}
     \bmeta(\bfx_1,\bfx_2) = \begin{pmatrix} x_{1,1} + x_{2,1} \\ x_{1,2} + x_{2,2} \\ x_{1,1}^2 + x_{2,1}^2 \\ x_{1,2}^2 + x_{2,2}^2 \\ x_{1,1}x_{1,2} + x_{2,1}x_{2,2} \end{pmatrix}.
 \end{equation*}
 Let
\begin{equation*}
    \begin{split}
        \calP_{z}:\quad\quad\ \ \ \bR^5\quad\quad\ \ &\rightarrow\quad\quad\ \bR^4, \\
     (z_1,z_2,z_3,z_4,w)& \mapsto (z_1,z_2,z_3,z_4),
    \end{split}
\end{equation*}
 be the projection map onto the first four coordinates. It is clear that  
 \begin{align*}
     \calP_z(\im(\bmeta)) & = \left\{(z_1,z_2,z_3,z_4)\in\bR^4 : 2 z_3\geq z_1^2,\ 2 z_4\geq z_2^2\right\} \\
     & = \left\{(z_1,z_3)\in\bR^2 : 2 z_3\geq z_1^2\right\} \times \left\{(z_2,z_4)\in\bR^2 : 2 z_4\geq z_2^2\right\}.
 \end{align*}
 Given $(z_1,z_2,z_3,z_4)\in \calP_z(\im(\bmeta))$, the behavior of $\im(\bmeta)$ can be divided into three cases:
 \begin{itemize}
     \item If $2 z_3 > z_1^2$ and $2 z_4 > z_2^2$, then $x_{1,1}\neq x_{2,1}$, $x_{2,1}\neq x_{2,2}$, and there is no singularity. Furthermore, there are two different $w$ such that $(z_1,z_2,z_3,z_4,w)\in \im(\bmeta)$.
     \item If exactly one of $2 z_3 = z_1^2$ and $2 z_4 = z_2^2$ holds, then only one of $x_{1,1} = x_{2,1}$ and $x_{2,1} = x_{2,2}$ is true, which still leads to non-singular behaviour but there only exists a single $w$ such that $(z_1,z_2,z_3,z_4,w)\in \im(\bmeta)$.
     \item If $2 z_3 = z_1^2$ and $2 z_4 = z_2^2$, then both $x_{1,1} = x_{2,1}$ and $x_{2,1} = x_{2,2}$ are true. There also exists a single $w$ such that $(z_1,z_2,z_3,z_4,w)\in \im(\bmeta)$, at which $\bmeta$ is singular.
 \end{itemize}
 We illustrate the above discussion in Figure~\ref{fig:geometry_eta}.

 \begin{figure}[htb!]
     \centering
    \includegraphics{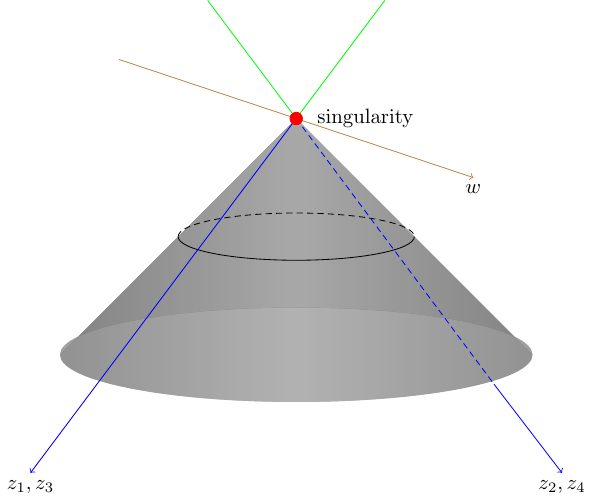}
    \caption{Geometry of $\im(\bmeta) = \bmeta((\bR^2)^2)$. The blue parts correspond to $\{2z_3\geq z_1^2\}$ and $\{2z_4\geq z_2^2\}$. The greens parts represent $\{2z_3 < z_1^2\}$ and $\{2z_4 < z_2^2\}$. The red part is the location of singularity, associated with the boundaries of both blue parts. 
    }
    \label{fig:geometry_eta}
 \end{figure}
 
As the result of the singularity of $\bmeta$, the regularity of $g$ may be worse than $f$. We give two simple examples with $n=2$ and $d=1$ below.

\begin{example}
 Consider $f(\bfx) = f(x_1,x_2) = |x_1| + |x_2|$. Let  $\Omega = [-1,1]$, $\bfx = (\epsilon, -\epsilon)$, and $\bfx' = (2\epsilon, -2\epsilon)$. Then $f(\bfx) = 2\epsilon$, $f(\bfx') = 4\epsilon$, and $\bmeta(\bfx) = (0, 2\epsilon^2)$, $\bmeta(\bfx') = (0, 8\epsilon^2)$. One has
 \begin{equation*}
     \lim_{\epsilon\rightarrow 0+} \frac{|g(\bmeta(\bfx)) - g(\bmeta(\bfx'))|}{\|\bmeta(\bfx) - \bmeta(\bfx')\|} = \lim_{\epsilon\rightarrow 0+}\frac{2\epsilon}{6\epsilon^2} = +\infty.
 \end{equation*}
Thus, $f$ is Lipschitz continuous while $g$ cannot be extended to a Lipschitz-continuous function on $\bR^2$. 
\end{example} 

\begin{example}
 Consider $f(\bfx) = f(x_1,x_2) = x_1^{4/3} + x_2^{4/3}$. Let  $\Omega = [-1,1]$, $\bfx = (\epsilon, -\epsilon)$, and $\bfx' = (0, 0)$. Then $f(\bfx) = 2\epsilon^{4/3}$, $f(\bfx') = 0$, and $\bmeta(\bfx) = (0, 2\epsilon^2)$, $\bmeta(\bfx') = (0, 0)$. One has
 \begin{equation*}
     \lim_{h\rightarrow 0+} \frac{g(0,h) - g(0,0)}{h} = \lim_{\epsilon\rightarrow 0}\frac{f(\epsilon, - \epsilon) - f(0,0)}{2\epsilon^2} = \lim_{\epsilon\rightarrow 0} \frac{2\epsilon^{4/3}}{2\epsilon^2} = +\infty.
 \end{equation*}
 Therefore, $f$ is $\mathcal{C}^1$, i.e., continuously differentiable, while $g$ cannot be extended to a $\mathcal{C}^1$ function on $\bR^2$.
\end{example}

\bibliographystyle{amsxport}
\bibliography{references}

\end{document}